\documentclass[journal,10pt]{IEEEtran}

\usepackage[colorlinks,linkcolor=black,anchorcolor=black,linktocpage=true,urlcolor=blue,citecolor=black]{hyperref}
\usepackage{stfloats}
\usepackage{booktabs}
\usepackage{algorithmic}
\usepackage{tikz}
\usetikzlibrary{arrows}
\usepackage{array}
\usepackage{graphicx}
\usepackage{url}
\usepackage{cite}
\usepackage{amsmath,amsthm,amsfonts,amssymb,mathrsfs,bm}
\allowdisplaybreaks[4]
\usepackage{tabularx,array,enumerate,pifont,tikz,epstopdf}

\newtheorem{theorem}{Theorem}
\newtheorem{lemma}{Lemma}

\newtheorem{remark}{Remark}

\newtheorem{assumption}{Assumption}

\begin{document}

\title{\LARGE \bf Fully Distributed Continuous-Time Algorithm for Nonconvex Optimization over Unbalanced Directed Networks}

\author{Jin~Zhang,
		Yahui~Hao,
        Lu~Liu,~\IEEEmembership{Senior Member, IEEE,}
        and~Haibo~Ji
\thanks{J. Zhang is with the Department of Automation, University of Science and Technology of China, Hefei 230027, China, and also with the Department of Biomedical Engineering, City University of Hong Kong, Hong Kong (e-mail: zj55555@mail.ustc.edu.cn).
	
	Y. Hao and L. Liu are with the Department of Biomedical Engineering, City University of Hong Kong, Hong Kong (e-mail: yahuihao2-c@my.cityu.edu.hk; luliu45@cityu.edu.hk).
	
	H. Ji is with the Department of Automation, University of Science and Technology of China, Hefei 230027, China (e-mail: jihb@ustc.edu.cn).}}


\maketitle

\begin{abstract}              
	This paper investigates the distributed continuous-time nonconvex optimization problem over unbalanced directed networks. The objective is to cooperatively drive all the agent states to an optimal solution that minimizes the sum of the local cost functions. Based on the topology balancing technique and adaptive control approach, a novel fully distributed algorithm is developed for each agent with neither prior global information concerning network connectivity nor convexity of local cost functions. By viewing the proposed algorithm as a perturbed system, its input-to-state stability with a vanishing perturbation is first established, and asymptotic convergence of the decision variables toward the optimal solution is then proved under the relaxed condition. A key feature of the algorithm design is that it removes the dependence on the smallest strong convexity constant of local cost functions, and the left eigenvector corresponding to the zero eigenvalue of the Laplacian matrix of unbalanced directed topologies. The effectiveness of the proposed fully distributed algorithm is illustrated with two examples.
\end{abstract}


\IEEEpeerreviewmaketitle

\section{Introduction}

\IEEEPARstart{D}{istributed} optimization problem (DOP) has experienced significant advances in the past decade because of its great potential in a wide range of applications. Typical examples of application include resource allocation, sensor networks, and power systems \cite{Molzahn2017survey,Ram2010distributed}. In the typical DOP of large-scale networks, each agent is often endowed with an individual local cost function. 
Seminal works on this topic primarily focus on discrete-time cases with convex local cost functions, which can be traced back to \cite{Tsitsiklis1984problems,Nedic2009distributed}. For more recent developments on distributed discrete-time convex optimization, one may refer to \cite{Nedic2014distributed,Xi2018linear,lu2020nesterov} and references therein.

In parallel, the distributed continuous-time convex optimization problem has also been extensively studied (see, for example, \cite{Wang2010control,Gharesifard2013distributed,Kia2015distributed,chen2020distributed,he2017continuous}) because many practical systems (e.g., unmanned vehicles and robots) operate in a continuous-time setting \cite{Yang2019survey}. A pioneering distributed gradient-based control approach was developed in \cite{Wang2010control} to address the DOP of multi-agent systems with single integrator dynamics over undirected graphs. In the subsequent work \cite{Gharesifard2013distributed}, the restriction on the communication network topologies was relaxed to balanced digraphs. By virtue of the proportional-integral control approach, a modified distributed Lagrangian-based algorithm was then proposed in \cite{Kia2015distributed} at the cost of special initialization so that communication and computation can be reduced. Other works that involve distributed continuous-time convex optimization, over either undirected or balanced directed networks, can be found in \cite{li2020distributed,li2021distributed} and references therein.

A central and standard assumption for the analysis of gradient-based minimization method in convex optimization is the convexity of the corresponding local cost functions \cite{bertsekas2009convex,Kia2015distributed}, which is exploited to not only refrain from the existence of local optima but also facilitate convergence analysis \cite{nesterov1998introductory}. However, it cannot be satisfied in a broad class of practical applications such as sparse approximations of images, matrix factorization, and compressed sensing \cite{bolte2014proximal}. In fact, cost functions in engineering practice are often nonconvex. More recently, there are increasing number of studies on distributed discrete-time nonconvex optimization providing that local cost functions satisfy additional but relaxed conditions \cite{scutari2019distributed,tatarenko2017non,engelmann2020decomposition}, such as the Polyak-Łojasiewicz (P-Ł) condition\cite{yi2021linear}, the $ \rho $-weakly convex condition\cite{chen2021distributed}, the $ \mu $-gradient dominated coondition \cite{tang2020distributed}, and the second order sufficiency condition \cite{farina2019distributed} among others. On the contrary, few works focus on developing continuous-time algorithms for distributed nonconvex optimization problems. Based on the canonical duality theory, a continuous-time algorithm was proposed in \cite{ren2021distributed} for a class of nonconvex optimization, but it can only be applied when undirected networks are considered.

The above-mentioned works, whether studying convex optimization or nonconvex optimization, assume that the concerned network topologies are undirected or balanced directed. It is of much more theoretical and practical significance to study unbalanced directed topologies as the information exchange between neighboring agents may be unidirectional due to limited bandwidth or other physical constraints. In the discrete-time case, by employing a row or a column stochastic matrix, several consensus-based strategies were proposed in \cite{Nedic2014distributed,Xi2018linear} to tackle unbalanced digraphs in distributed convex optimization problem. These strategies are then extented to address the more challenging scenario of distributed nonconvex optimization over unbalanced digraphs \cite{tatarenko2017non,chen2021distributed}. However, the agents in those works were required to know their out-degree \cite{Nedic2014distributed,tatarenko2017non}, which is a form of global information. 

In the continuous-time case, a new distributed control strategy was developed based on the topology balancing technique to tackle the distributed convex optimization over unbalanced directed networks \cite{Li2017distributed}. Nevertheless, it cannot be adopted when the left eigenvector corresponding to the zero eigenvalue of the concerned Laplacian matrix is not available in advance. For the same problem, a distributed estimator was designed in \cite{Zhu2018continuous} to remove the explicit dependence on the left eigenvector, and the gradient term therein was divided by the state of the distributed estimator. However, the control gains of the algorithm proposed in \cite{Zhu2018continuous} still involve certain global information concerning the network connectivity such as the second smallest eigenvalue of the Laplacian matrix, and the smallest strong convexity constant of local cost functions. To the best of our knowledge, no distributed \textit{continuous-time} algorithm has been proposed to address the distributed \textit{nonconvex} optimization over unbalanced directed networks up till now.
To sum up, the existing algorithms that tackle unbalanced digraphs more or less rely on global information concerning the network connectivity and/or cost functions, and are thus not \textit{fully distributed}.

Motivated by the above observations, this paper aims at developing a fully distributed continuous-time algorithm to address the distributed nonconvex optimization over unbalanced directed networks. The main challenge lies in establishing asymptotic convergence of the agent states in the absence of symmetric Laplacian matrix and the convexity of local cost functions. A novel algorithm is developed over unbalanced directed network topologies based on the topology balancing technique \cite{Bullo2019lectures,Zhu2018continuous} and adaptive control approach \cite{Li2017distributed,ioannou2012robust}. The developed continuous-time algorithm is fully distributed in the sense that it does not depend on any global information about the network connectivity or the local cost functions. The main contributions of this paper are summarized as follows.

1) In contrast with those works addressing the DOP for undirected graphs or balanced digraphs \cite{Wang2010control,Gharesifard2013distributed,Kia2015distributed,li2020distributed,li2021distributed,Wu2018distributed}, this work considers unbalanced digraphs that are more general and also more challenging. Contrary to the algorithms in \cite{Zhu2018continuous,Li2017distributed,tatarenko2017non,chen2021distributed} that deal with unbalanced digraphs, our distributed adaptive continuous-time algorithm does not depend on any global information concerning network connectivity or cost functions, and can be applied in a fully distributed manner. Specifically, the left eigenvector corresponding to the zero eigenvalue of the Laplacian matrix, which plays a crucial role in \cite{Li2017distributed}, no longer needs to be known \textit{a priori}. The algorithm proposed in this work is thus expected to be adopted in a wider range of applications.

2) The requirement on the convexity of the local cost functions, which are exploited to facilitate the convergence analysis in \cite{Wang2010control,Gharesifard2013distributed,Kia2015distributed,li2020distributed,li2021distributed,Zhu2018continuous}, is removed. Instead, in this work, the local cost functions are allowed to be nonconvex. Such a relaxed condition greatly broadens the application scope of distributed optimization.

The layout of this paper is as follows. Section \ref{section preliminaries} reviews some preliminaries on graph theory and convex analysis, and provides the problem formulation and control objective of this paper. Section \ref{section main results} presents the algorithm design and convergence analysis by means of adaptive control. Section \ref{section simulation results} provides two examples to illustrate the effectiveness of the proposed algorithm, which is followed by the conclusion and future challenge in Section \ref{section conclusion}.

\textit{Notation}: Let $\mathbb{R}$, $\mathbb{R}^{n}$ and $ \mathbb{R}^{N \times N} $ be the sets of real numbers, $ n $-order real vectors and $ N $-dimensional real square matrices, respectively. $I_{n}$ refers to the $ n $-dimensional identity matrix. Let $\mathbf{0}_{n}$ and $\mathbf{1}_{n}$, or simply $\mathbf{0}$ and $\mathbf{1}$, represent the $ n $-dimensional column vector in which all entries are equal to $ 0 $ and $ 1 $, respectively. $ A_{i} $ and $ A_{i}^{j} $ denote the $ i $th row elements and the $ (i,j) $ entry of matrix $ A $, respectively. The transpose of vector $ x $ and matrix $ A $ are denoted by $x^{\mathrm{T}}$ and $A^{\mathrm{T}}$, respectively. $\|\cdot\|$ represents the Euclidean norm of vectors or induced 2-norm of matrices. The Kronecker product of matrices $ A $ and $ B $ is represented by $ A\otimes B $. $ \operatorname{col}\left(x_{1}, x_{2}, \ldots, x_{n}\right) $ and $\operatorname{diag}\left(x_{1}, x_{2}, \ldots, x_{n}\right)$ represent a column vector and a diagonal matrix, respectively, with $x_{1}, x_{2}, \ldots, x_{n}$ being their elements.

\section{Preliminaries and Problem Formulation}\label{section preliminaries}
In this section, we first present some preliminaries on graph theory and convex analysis, and then formulate the problem.

\subsection{Graph Theory}
A \textit{directed graph} (in short, a \textit{digraph}) of order $ N $ can be described by a triplet $\mathcal{G}=(\mathcal{V}, \mathcal{E}, \mathcal{A})$, where $\mathcal{V}=\{1, \ldots, N\}$ is a set of nodes, $\mathcal{E} \subseteq \mathcal{V} \times \mathcal{V}$ is a collection of edges, and an adjacency matrix $\mathcal{A}$. 
For $ i,j\in\mathcal{V} $, the ordered pair $(j, i) \in \mathcal{E}$ refers to an edge from $ j $ to $ i $.
A directed path in a digraph is an ordered sequence of nodes, in which any pair of consecutive nodes is a directed edge. 
A digraph is said to be \textit{strongly connected} if for each node, there exists a directed path from any other node to itself. 
The adjacency matrix is defined as $\mathcal{A}=\left[a_{i j}\right] \in \mathbb{R}^{N \times N}$, where $a_{i i}=0$ for all $ i $, $a_{i j}>0$ if $(j, i) \in \mathcal{E}$, otherwise $a_{i j}=0 $. The Laplacian matrix $\mathcal{L}=\left[l_{i j}\right] \in \mathbb{R}^{N \times N}$ associated with the digraph $\mathcal{G}$ is defined as $l_{i i}=\sum_{j=1}^{N} a_{i j}$ and $l_{i j}=-a_{i j}$ for $i \neq j $. A digraph $\mathcal{G}$ is called \textit{balanced} if and only if $\mathbf{1}_{N}^{\mathrm{T}} \mathcal{L}=\mathbf{0}_{N}^{\mathrm{T}}$, otherwise it is called \textit{unbalanced}. One can consult \cite{Bullo2019lectures} for more details. 

\begin{lemma}(see \cite{Olfati-Saber2004consensus,Mei2015distributed}) \label{graph theory lemma}
	Let $ \mathcal{L} $ be the Laplacian matrix associated with a strongly connected digraph $ \mathcal{G} $. Then the following statements hold.
	\begin{itemize}
		\item[\romannumeral1)] \label{graph theory lemma_1} There exists a left eigenvector $ \xi=(\xi_{1},\xi_{2},\ldots,\xi_{N})^{\mathrm{T}}  $ associated with the zero eigenvalue of the Laplacian matrix such that $ \sum_{i=1}^{N}\xi_{i}=1 $, $ \xi_{i}>0, ~i=1,2,\ldots,N $, and $ \xi^{\mathrm{T}}\mathcal{L}=\mathbf{0}_{N}^{\mathrm{T}} $.
		\vskip 0.1cm
		\item[\romannumeral2)] \label{graph theory lemma_2} Define $\bar{\mathcal{L}}=\ \mathcal{R}\mathcal{L}+\mathcal{L}^{\mathrm{T}}\mathcal{R} $ with $ \mathcal{R}=\mathrm{diag}\left(\xi_{1},\xi_{2},\ldots,\xi_{N} \right) $. Then 
		$\min_{\varsigma^{\operatorname{T}}x=0,x\neq0}x^{\operatorname{T}}\bar{\mathcal{L}}x > \frac{\lambda_{2}(\bar{\mathcal{L}})}{N}x^{\operatorname{T}}x $ for any positive vector $ \varsigma $, where $ \lambda_{2}(\bar{\mathcal{L}}) $ represents the second smallest eigenvalue of $ \bar{\mathcal{L}} $.
		\vskip 0.1cm
		\item[\romannumeral3)] \label{graph theory lemma_3} $ e^{-\mathcal{L}t} $ is a nonnegative matrix with positive
		diagonal entries for all $ t>0 $, and $ \lim_{t\to\infty}e^{-\mathcal{L}t}= \mathbf{1}_{N}\xi^{\mathrm{T}} $.
	\end{itemize}
\end{lemma}

\subsection{Convex Analysis}
This subsection reviews the definitions of convexity and Lipschitz continuity. One can refer to \cite{bertsekas2009convex,khalil2002nonlinear} for more details.

A subset $\varOmega$ of $\mathbb{R}^{n}$ is called convex if, for all $ \alpha \in[0,1] $,
$$
\alpha x+(1-\alpha) y \in \varOmega, \quad \forall x, y \in \varOmega.
$$
A function $f: \varOmega\subset\mathbb{R}^{n} \mapsto \mathbb{R}$ is called \textit{convex} if, for all $ \alpha \in[0,1] $,
\begin{equation} \label{convex}
	f(\alpha x+(1-\alpha) y) \leq \alpha f(x)+(1-\alpha) f(y), \quad \forall x, y \in \varOmega,
\end{equation}
otherwise it is called \textit{nonconvex}. A convex function $f: \varOmega \mapsto \mathbb{R}$ is called \textit{strictly convex} if, for all $ \alpha \in(0,1) $, the inequality (\ref{convex}) is strict for all $ x,y\in\varOmega $ with $ x\neq y $.
A convex function $f: \varOmega \mapsto \mathbb{R}$ is called \textit{strongly convex} if there exists a positive constant $m $ such that $(x-y)^{T}(\nabla f(x)-\nabla f(y)) \geq m\|x-y\|^{2}$ for all $x, y \in \varOmega $, where $\nabla f$ denotes the gradient.

A function $g: \mathbb{R}^{n} \mapsto \mathbb{R}^{n}$ is said to be \textit{Lipschitz continuous}, or simply \textit{Lipschitz}, if there exists a constant $l>0$ such that the following Lipschitz condition is satisfied,
\begin{equation} \label{Lipschitz condition}
	\|g(x)-g(y)\| \leq l\|x-y\|, \quad \forall x, y \in \mathbb{R}^{n}.
\end{equation}

\subsection{Problem Formulation} \label{section problem formulation}

Consider a multi-agent system composed of $ N $ identical agents over an unbalanced directed network. Each agent $ i $ is attached with a private local cost function $ f_{i}(s): \mathbb{R}^{n}\to \mathbb{R} $, where $ s\in \mathbb{R}^{n} $ represents the local decision variable. The global cost function and the corresponding optimal solution are defined as $ f(s)=\sum_{i=1}^{N}f_{i}(s) $ and $ s^{\star} $, respectively. In this work, the objective is to design a fully distributed continuous-time algorithm in the sense that neither strong convexity of local cost functions nor global information about network connectivity is required, such that the following DOP can be solved,
\begin{equation} \label{problem}
	\min_{s\in\mathbb{R}^{n}}~ f(s).
\end{equation}

To solve the problem, necessary assumptions are introduced.

\begin{assumption} \label{graph assumption}
	The digraph $ \mathcal{G} $ is strongly connected.
\end{assumption}

\begin{remark}
	In contrast with those works \cite{Wang2010control,Gharesifard2013distributed,Kia2015distributed,li2020distributed,li2021distributed,Wu2018distributed} that address the distributed continuous-time optimization for undirected graphs or balanced digraphs, this paper concentrates on the more general and also more challenging case of unbalanced directed communication topologies. 
	Unbalanced topologies bring challenges to the convergence establishment of the agent states toward the exact optimal solution because consensus cannot be achieved with the unweighted gradient information transmitted by the asymmetric topologies.
\end{remark}

\begin{assumption} \label{assumption_cost functions}
	Each local cost function $ f_{i} $ is differentiable, and its gradient $ \nabla f_{i} $ is globally Lipschitz on $ \mathbb{R}^{n} $ with constant $ l_{i} $. The optimal solution $ s^{*}\in \mathbb{R}^{n} $ to the DOP in (\ref{problem}) exists and is unique.
\end{assumption}

\begin{remark} \label{remark_cost function_1}
	In most existing works that study the DOP  \cite{Wang2010control,Gharesifard2013distributed,Kia2015distributed,li2020distributed,li2021distributed,Zhu2018continuous}, the local cost functions are assumed to be strongly convex. Under such a stringent assumption, the admissible local cost functions are confined to those which can be bounded from the below by a quadratic function. In this work, on the contrary, the local cost functions are allowed to take nonconvex forms, and this greatly enlarges the set of admissible cost functions and thus broadens the scope of application.
\end{remark}

\begin{remark} 
   The existence and uniqueness of the optimal solution can be guaranteed when the global cost function is strictly convex or the set of global minimizers is a singleton\cite{bertsekas2009convex,bolte2014proximal}. The differentiability of local cost functions is required only to facilitate the convergence analysis. Illustrative Example 2 shows that our proposed adaptive algorithm can still solve the concerned problem even if the cost function is nondifferentiable.
\end{remark}

\section{Main Result}\label{section main results}
In this section, we propose a fully distributed adaptive algorithm to solve the DOP in (\ref{problem}) over unbalanced directed networks without any global information concerning the network connectivity or local cost functions.

\subsection{Fully Distributed Algorithm}

In this subsection, based on adaptive control approach, a fully distributed algorithm is designed for each agent $ i,~i=1,2,\ldots,N $ as follows,
\begin{equation} \label{algorithm}
	\small
	\left\{
	\begin{aligned}
		\dot{x}_{i} = & -\tfrac{\nabla f_{i}(x_{i})}{w_{i}^{i}} - (\sigma_{i}+\rho_{i})\sum_{j=1}^{N}a_{ij}(x_{i}-x_{j}) - \sum_{j=1}^{N}a_{ij}(v_{i}-v_{j}),  \\[-1mm]
		\dot{v}_{i} = &	(\sigma_{i}+\rho_{i})\sum_{j=1}^{N}a_{ij}(x_{i}-x_{j}), \\[-1mm]
		\dot{w}_{i} = & -\sum_{j=1}^{N} a_{i j}\left(w_{i}-w_{j} \right), \\[-1mm]
		\dot{\sigma}_{i} = & \Big(\sum_{j=1}^{N}a_{ij}(x_{i}-x_{j})\Big)^{\operatorname{T}}\Big(\sum_{j=1}^{N}a_{ij}(x_{i}-x_{j})\Big), 
	\end{aligned}
	\right.
\end{equation}
where $ x_{i}\in \mathbb{R}^{n} $ is the state of agent $ i $, $ v_{i}\in \mathbb{R}^{n} $ and $ w_{i}\in \mathbb{R}^{N}$ are two auxiliary variables, $ w_{i}^{i} $ is the $ i $th component of $ w_{i} $, and the initial value $ w_{i}(0) $ satisfies $ w_{i}^{i}(0)=1 $, otherwise $ w_{i}^{k}(0)=0 $ for all $ k\neq i $; $ \sigma_{i} $ is an adaptive gain with initial condition $ \sigma_{i}(0)>0 $, and the dynamic gain $ \rho_{i} $ is designed as $ \rho_{i} = \big(\sum_{j=1}^{N}a_{ij}(x_{i}-x_{j})\big)^{\operatorname{T}}\big(\sum_{j=1}^{N}a_{ij}(x_{i}-x_{j})\big) $.

Define $ w=\mathrm{col}(w_{1},w_{2},\ldots,w_{N}) $. It follows that $ \dot{w} = -\left(\mathcal{L}\otimes I_{N} \right)w $. 
By referring to $ w_{i}^{i}(0)=1 $, and $ w_{i}^{k}(0)=0 $ for all $ k\neq i $ as well as \romannumeral3) of Lemma \ref{graph theory lemma}, one has
\begin{align} \label{w_i^i}
	w_{i}^{i}(t) = & \big(e^{-(\mathcal{L}\otimes I_{N})t}\big)_{(i-1)N+i}\cdot w(0) \notag\\
	= & \big(e^{-(\mathcal{L}\otimes I_{N})t}\big)_{(i-1)N+i}^{(i-1)N+i}\cdot w_{i}^{i}(0) >0,
\end{align}	
for all $ t \geq 0 $. Therefore, the term $ -\tfrac{\nabla f_{i}(x_{i})}{w_{i}^{i}} $ in algorithm (\ref{algorithm}) is well defined. Moreover, by applying \romannumeral3) of Lemma \ref{graph theory lemma} and recalling the initial condition of $ w(0) $, we can obtain that 
\begin{equation} \label{w_i^i_1}
	\begin{aligned}
		\lim_{t\to\infty}w(t)= &\lim_{t\to\infty}e^{-(\mathcal{L}\otimes I_{N})t}w(0)\\
		=& \left( \mathbf{1}_{N}\xi^{\mathrm{T}}\otimes I_{N}\right)w(0)=\mathbf{1}_{N}\otimes \xi.
	\end{aligned}
\end{equation}
This implies that $ w_{i}^{i}(t), ~i=1,2,\ldots,N $ are bounded for all $ t>0 $.

\begin{remark} \label{remark_z_i}
	Compared to the algorithm proposed in \cite{Li2017distributed}, we use 
	$ \tfrac{\nabla f_{i}(x_{i})}{w_{i}^{i}} $ instead of $ \tfrac{\nabla f_{i}(x_{i})}{\xi_{i}} $ in the adaptive algorithm (\ref{algorithm}) so that it is able to not only tackle the imbalance caused by unbalanced directed topologies but also eliminate the restrictive requirement on the exact value of the left eigenvector corresponding to the zero eigenvalue of the Laplacian matrix. The algorithm (\ref{algorithm}) does not depend on any global information concerning network connectivity or cost functions, and is thus fully distributed. 
\end{remark}

\begin{remark}
	The design of algorithm (\ref{algorithm}) takes the modified Lagrange structure in \cite{Kia2015distributed} with an adaptive control scheme. 
	By using the adaptive gain $ \sigma_{i} $ and the dynamic gain $ \rho_{i} $ instead of constant gains as in \cite{Kia2015distributed,Zhu2018continuous}, some global information concerning cost functions and network connectivity is no longer needed, such as the smallest strong convexity constant of local cost functions, and the second smallest or largest eigenvalue of the Laplacian matrix. 
\end{remark}

Define $$ x=\operatorname{col}(x_{1},x_{2},\ldots,x_{N}), \quad v=\operatorname{col}(v_{1},v_{2},\ldots,v_{N}), $$ 
$$ \mathcal{C}=\operatorname{diag}(\sigma_{1},\sigma_{2},\ldots,\sigma_{N}), \quad \mathcal{B}=\operatorname{diag}(\rho_{1},\rho_{2},\ldots,\rho_{N}), $$
$$ \nabla\tilde{f}(x)=\operatorname{col}\big(\nabla f_{1}(x_{1}),\nabla f_{2}(x_{2}),\ldots,\nabla f_{N}(x_{N})\big), $$ 
$$ \mathcal{W}^{-1}=\operatorname{diag}(\tfrac{1}{w_{1}^{1}},\tfrac{1}{w_{2}^{2}},\ldots,\tfrac{1}{w_{N}^{N}}), \quad e_{i}= \sum_{j=1}^{N}a_{ij}(x_{i}-x_{j}). $$ 
It can be seen from (\ref{w_i^i}) that the matrix $ \mathcal{W}^{-1} $ is well defined. Given any $ \sigma_{i}(0)>0 $, it can be proved that the adaptive gain $ \sigma_{i}(t) $ remains to be positive for all $ t>0 $. Thus, the dynamics of $ (x, v, w, \sigma_{i}) $ can be written in the following form,
\begin{equation} \label{compact form_1}
	\small
	\left\{
	\begin{aligned} 
		\dot{x} = & -(\mathcal{W}^{-1}\otimes I_{n})\nabla \tilde{f}(x) - \big((\mathcal{C}+\mathcal{B})\mathcal{L}\otimes I_{n}\big)x 	- (\mathcal{L}\otimes I_{n})v, \\
		\dot{v} = &	\big((\mathcal{C}+\mathcal{B})\mathcal{L}\otimes I_{n}\big)x, \\
		\dot{w} = & -\left(\mathcal{L}\otimes I_{N} \right)w, \quad
		\dot{\sigma}_{i} = e_{i}^{\operatorname{T}}e_{i}, 
	\end{aligned}
	\right.
\end{equation}

In what follows, our goal is to show that the agent states $ x_{i}, ~i=1,2,\ldots,N $ of (\ref{algorithm}) will converge to the optimal solution  $  s^{\star} $ of the DOP in (\ref{problem}).

\subsection{Convergence Analysis} \label{subsection lemma}

To proceed, a preliminary result on the optimality condition will be first established. 
Define $ \mathcal{R}^{-1}=\operatorname{diag}(\frac{1}{\xi_{1}},\frac{1}{\xi_{2}},\ldots,\frac{1}{\xi_{N}})$, where $ \xi_{i}, ~i=1,2,\ldots,N $ is the $ i $th component of the left eigenvector $ \xi=\left(\xi_{1},\xi_{2},\ldots,\xi_{N} \right)^{\mathrm{T}}  $. The following lemma reveals the optimality condition in terms of a set of equations that the optimal solution $ s^{\star} $ satisfies.

\begin{lemma} \label{optimality_condition}
	Under Assumptions \ref{graph assumption}--\ref{assumption_cost functions}, suppose that the point $ (\bar{x},\bar{v}) $ satisfies the following equations,
	\begin{subequations}\label{equilibrium point_r}
		\small
		\begin{align}
			\mathbf{0} = & \!-\!(\mathcal{R}^{-1}\!\otimes I_{n})\nabla \tilde{f}(\bar{x}) \!-\! \big((\mathcal{C}+\mathcal{B})\mathcal{L}\otimes I_{n}\big)\bar{x} 	\!-\! (\mathcal{L}\otimes I_{n})\bar{v}, \label{equilibrium point_r_a}\\
			\mathbf{0} = & \big((\mathcal{C}+\mathcal{B})\mathcal{L}\otimes I_{n}\big)\bar{x}.\label{equilibrium point_r_b}
		\end{align}
	\end{subequations}
	Then, one has $ \bar{x}=\mathbf{1}_{N}\otimes s^{\star} $, with $ s^{\star} $ being the optimal solution of the DOP in (\ref{problem}).
\end{lemma}

\begin{proof}
	Define $ \bar{x}=\operatorname{col}(\bar{x}_{1}, \bar{x}_{2}, \ldots, \bar{x}_{N}) $. It can be seen from (\ref{equilibrium point_r_b}) that $ \bar{x}=\mathbf{1}_{N}\otimes q $ holds for some vector $ q \in \mathbb{R}^{n} $. On the one hand, by pre-multiplying both sides of equation (\ref{equilibrium point_r_a}) with $ \xi^{\mathrm{T}} \otimes I_{n} $, it follows from i) of Lemma \ref{graph theory lemma} that $ \sum_{i=1}^{N}\nabla f_{i}(\bar{x}_{i})=\mathbf{0}_{n} $. On the other hand, under Assumption \ref{assumption_cost functions}, the optimality condition $ \sum_{i=1}^{N}\nabla f_{i}(s^{\star})=\mathbf{0}_{n} $ is satisfied. Therefore, it can be obtained that $ \bar{x}=\mathbf{1}_{N}\otimes s^{\star} $.
\end{proof}

With Lemma \ref{optimality_condition} in hand, to prove that the agent states $ x_{i}, ~i=1,2,\ldots,N $ of (\ref{algorithm}) converge to the optimal solution $ s^{\star} $ of the DOP in (\ref{problem}), it is sufficient to show that $ (x,v) $ of (\ref{compact form_1}) converges to $ (\bar{x},\bar{v}) $ in Lemma \ref{optimality_condition}. 

To proceed, define $ \tilde{x}=x-\bar{x} $ and $ \tilde{v}=v-\bar{v} $. By subtracting the equations (\ref{compact form_1}) from (\ref{equilibrium point_r}), and noting that $ \mathcal{W}^{-1}\neq \mathcal{R}^{-1} $, the dynamics of $ (\tilde{x}, \tilde{v}, w, \sigma_{i}) $ can be written in the following form,
\begin{subequations} \label{compact_form_2}
	\small
	\begin{align}
		\dot{\tilde{x}} = & -\big(\mathcal{W}^{-1}\otimes I_{n}\big)h  + \big((\mathcal{R}^{-1} -\mathcal{W}^{-1}) \otimes I_{n}\big)\nabla \tilde{f}(\bar{x}) \notag \\
		& - \big((\mathcal{C}+\mathcal{B})\mathcal{L}\otimes I_{n}\big)\tilde{x} - (\mathcal{L}\otimes I_{n})\tilde{v}, \label{compact_form_2_a}\\
		\dot{\tilde{v}} = &	\big((\mathcal{C}+\mathcal{B})\mathcal{L}\otimes I_{n}\big)\tilde{x}, \label{compact_form_2_b}\\
		\dot{w} = & -\left(\mathcal{L}\otimes I_{N} \right)w, \quad
		\dot{\sigma}_{i} = e_{i}^{\operatorname{T}}e_{i}, \label{compact_form_2_c}
	\end{align} 
\end{subequations} 
where $ h= \nabla \tilde{f}(\bar{x}+\tilde{x}) - \nabla \tilde{f}(\bar{x}) $. Therefore, to show that $ \lim_{t\to\infty}x_{i}(t)=s^{\star} $, by $ \tilde{x}=x-\bar{x} $ in (\ref{compact_form_2}) and $ \bar{x}=\mathbf{1}_{N}\otimes s^{\star} $ in Lemma \ref{optimality_condition}, we only need to prove that the state $ \tilde{x} $ of (\ref{compact_form_2}) coverges to zero as time tends to infinity. 

However, the origin $ (\tilde{x}, \tilde{v})=(\mathbf{0}, \mathbf{0}) $ is not the equilibrium point of (\ref{compact_form_2}) as $ \mathcal{R}^{-1}\neq \mathcal{W}^{-1} $. This brings some extra challenge to the convergence analysis. To tackle this issue, we need to introduce some coordinate transformations. Define two new variables $ \zeta = (\mathcal{L}\otimes I_{n})\tilde{x} $ and $ \eta = (\mathcal{L}\otimes I_{n})\tilde{v} $. In what follows, we will first prove that $ \lim_{t\to\infty} \zeta(t)=\mathbf{0} $ and $ \lim_{t\to\infty} \eta(t)=\mathbf{0} $, which are followed by $ \lim_{t\to\infty}\tilde{x}(t)= \mathbf{1}_{N}\otimes \tau_{1}$ and $ \lim_{t\to\infty}\tilde{v}(t)= \mathbf{1}_{N}\otimes \tau_{2}$, for two constant vectors $ \tau_{1},\tau_{2}\in\mathbb{R}^{n} $. Then, we will show that $ \tau_{1}=\mathbf{0} $ and $ \tau_{2}<\infty $ by seeking a contradiction.

Now we are ready to present the main result of this work.

\begin{theorem} \label{theorem1}
	Suppose Assumptions \ref{graph assumption}--\ref{assumption_cost functions} hold. For $ i=1,2,\ldots,N $, let $ \sigma_{i}(0)>0 $, $ w_{i}^{i}(0)=1 $, and $ w_{i}^{k}(0)=0 $ for all $ k\neq i $. Then, for any initial conditions $ x_{i}(0) $ and $ v_{i}(0) $, the DOP in (\ref{problem}) is solved by the fully distributed algorithm (\ref{algorithm}).
\end{theorem}

\begin{proof}
	To prove Theorem \ref{theorem1}, it is sufficient to prove that the state $ \tilde{x} $ of (\ref{compact_form_2}) will converge to zero as time tends to infinity in the case of $ \mathcal{R}^{-1}\neq \mathcal{W}^{-1} $. The proof is composed of the following two parts.
	\vskip 0.15cm
	\noindent\textbf{Part 1.} Show that $ \lim_{t\to\infty} \zeta(t)=\mathbf{0} $ and $ \lim_{t\to\infty} \eta(t)=\mathbf{0} $.
	
	To proceed, let $ \zeta_{i}\in\mathbb{R}^{n},~ i=1,2,\ldots,N $ be a column vector stacked from the $ ((i-1)\times n+1) $th element to the $ (i\times n) $th element of vector $ \zeta $. Recalling that $ e_{i}= \sum_{j=1}^{N}a_{ij}(x_{i}-x_{j}) $, simple derivation gives $ e_{i} = \zeta_{i} $. Then, the dynamics of $ (\zeta, \eta, w, \sigma_{i}) $ can be written as follows,
	\begin{equation} \label{Laplacian dynamics_r}
		\small
		\left\{
		\begin{aligned} 
			\dot{\zeta} = & -\big(\mathcal{L}\mathcal{W}^{-1}\otimes I_{n}\big)h + \big(\mathcal{L}(\mathcal{R}^{-1} -\mathcal{W}^{-1}) \otimes I_{n}\big)\nabla \tilde{f}(\bar{x})  \\
			& - \big(\mathcal{L}(\mathcal{C}+\mathcal{B})\otimes I_{n}\big)\zeta - (\mathcal{L}\otimes I_{n})\eta, \\
			\dot{\eta} = &	\big(\mathcal{L}(\mathcal{C}+\mathcal{B})\otimes I_{n}\big)\zeta, \\
			\dot{w} = & -\left(\mathcal{L}\otimes I_{N} \right)w, \quad
			\dot{\sigma}_{i} = \zeta_{i}^{\operatorname{T}}\zeta_{i}.
		\end{aligned}
		\right.
	\end{equation}
	Define $ \chi = \operatorname{col}(\zeta, \eta) $. Then the dynamics of $ \chi $ can be rewritten as
	\begin{equation} \label{tilde_x_new dynamics}
		\dot{\chi}= \varphi(\chi) + \phi(t),
	\end{equation}
	where $ \varphi(\chi) $ and $ \phi(t) $ are defined in (\ref{phi_chi}) on the next page. At first, we will establish the asymptotical convergence to the origin of $ \chi $ in system (\ref{tilde_x_new dynamics}) with $ \mathcal{R}^{-1}\neq \mathcal{W}^{-1} $ and $ \phi(t)\neq 0 $. 
	\newcounter{mytempeqncnt}
	\begin{figure*}[ht]
		\vspace*{-12pt}
		\setcounter{mytempeqncnt}{\value{equation}}
		\begin{align} \label{phi_chi}
			\varphi(\chi) = \left(\begin{array}{c}
				-(\mathcal{L}\mathcal{W}^{-1}\otimes I_{n})h - \big(\mathcal{L}(\mathcal{C}+\mathcal{B})\otimes I_{n}\big)\zeta - (\mathcal{L}\otimes I_{n})\eta \\
				\big(\mathcal{L}(\mathcal{C}+\mathcal{B})\otimes I_{n}\big)\zeta 
			\end{array}\right),  \quad
			\phi(t) = \left(\begin{array}{c}
				\big(\mathcal{L}(\mathcal{R}^{-1} -\mathcal{W}^{-1}) \otimes I_{n}\big)\nabla \tilde{f}(\bar{x}) \\
				\mathbf{0}
			\end{array}\right) 
		\end{align}
		\hrule
		\hrulefill
	\end{figure*}
	
	Under Assumption \ref{assumption_cost functions}, $ h= \nabla \tilde{f}(\bar{x}+\tilde{x}) - \nabla \tilde{f}(\bar{x}) $ and thus $ \varphi(\chi) $ are locally Lipschitz in $ \chi\in\varOmega $ for any compact subset $ \varOmega\in\mathbb{R}^{n} $. In addition, it follows from (\ref{w_i^i_1}) that $ \phi(t) $ is bounded for all $t \geq 0$.
	According to Theorem 4.19 in \cite{khalil2002nonlinear}, to establish the asymptotical convengence to the origin of $ \chi $ in (\ref{tilde_x_new dynamics}), global uniform asymptotical stability of the unperturbed system $ \dot{\chi}= \varphi(\chi) $ should be established at first, and input-to-state stability (ISS) of the system (\ref{tilde_x_new dynamics}) will be presented in turn.
	The proof can be accomplished by two steps.
	
	\textit{Step 1.} Show that the equilibrium point $\chi=\mathbf{0}$ of the unperturbed system $ \dot{\chi}= \varphi(\chi) $ is globally uniformly asymptotically stable. By referring to \cite{bernstein2009matrix}, one can obtain that $ \operatorname{rank}(\mathcal{L}^{\operatorname{T}}\mathcal{L})=\operatorname{rank}(\mathcal{L})=N-1 $. Thus, zero is a simple eigenvalue of matrix $ \mathcal{L}^{\operatorname{T}}\mathcal{L} $. Note that $ \xi_{i} $ and $ \sigma_{i} $ are positive for all $ i=1,2,\ldots,N $. Consider the following Lyapunov function candidate,
	\begin{equation} \label{Lyapunov function candidate}
		V=V_{1}+V_{2}+\tfrac{33N\bar{\lambda}_{N}(\mathcal{L}^{\operatorname{T}}\mathcal{L})}{\lambda_{2}(\bar{\mathcal{L}})^{2}}V_{3},
	\end{equation}
	where $ \bar{\lambda}_{N}(\mathcal{L}^{\operatorname{T}}\mathcal{L}) $ denotes the largest eigenvalue of $ \mathcal{L}^{\operatorname{T}}\mathcal{L} $, $ \lambda_{2}(\bar{\mathcal{L}}) $ denotes the second smallest eigenvalue of $ \bar{\mathcal{L}}=\mathcal{R}\mathcal{L}+\mathcal{L}^{\operatorname{T}}\mathcal{R} $, and
	\begin{equation*} \label{Laplacian dynamics_r}
		\left\{
		\begin{aligned} 
			~V_{1} = & \tfrac{1}{2}\sum_{i=1}^{N}(\sigma_{i}-\sigma_{0})^{2}, \\[-1mm]
			~V_{2} = & \tfrac{1}{2}\sum_{i=1}^{N}\xi_{i}(2\sigma_{i}+\rho_{i})\zeta_{i}^{\operatorname{T}}\zeta_{i}, \\[1mm]
			~V_{3} = &\tfrac{1}{2}(\zeta+\eta)^{\mathrm{T}}(\mathcal{R}\otimes I_{n})(\zeta+\eta),
		\end{aligned}
		\right.
	\end{equation*}
	with $ \sigma_{0} $ being a positive constant to be determined later. 

The derivatives of $ V_{1} $ and $ V_{2} $ along the trajectories of the unperturbed system $ \dot{\chi}= \varphi(\chi) $ satisfy
\begin{align}
	\dot{V}_{1} =& \sum_{i=1}^{N}\zeta_{i}^{\operatorname{T}}(\sigma_{i}-\sigma_{0})\zeta_{i}, \label{final_derivative of V_1}\\
	\dot{V}_{2} = & 2\sum_{i=1}^{N}\xi_{i}(\sigma_{i}+\rho_{i})\zeta_{i}^{\operatorname{T}}\dot{\zeta}_{i} + \sum_{i=1}^{N}\xi_{i}\rho_{i}\zeta_{i}^{\operatorname{T}}\zeta_{i}. \label{final_derivative of V_2}
\end{align}
By combining (\ref{final_derivative of V_1})--(\ref{final_derivative of V_2}), and recalling $ \varphi(\chi) $ in (\ref{phi_chi}), we can obtain that
{\small \begin{align} \label{derivative3 of V_1+V_2}
	\dot{V}_{1} + \dot{V}_{2} \leq  & -2 \zeta^{\operatorname{T}}\big((\mathcal{C}+\mathcal{B})\mathcal{R}\mathcal{L}(\mathcal{C}+\mathcal{B})\otimes I_{n}\big)\zeta + \zeta^{\operatorname{T}}\big((\mathcal{C}+\mathcal{R}\mathcal{B} \notag\\[2mm]
	& -\sigma_{0}I_{N}) \otimes I_{n}\big)\zeta -2 \zeta^{\operatorname{T}}\big((\mathcal{C}+\mathcal{B})\mathcal{R}\mathcal{L}\mathcal{W}^{-1}\otimes I_{n}\big)h \notag\\[2mm]
	& -2 \zeta^{\operatorname{T}}\big((\mathcal{C}+\mathcal{B})\mathcal{R}\mathcal{L}\otimes I_{n}\big)\eta. 
\end{align}}

Define $ \hat{\zeta} = \big((\mathcal{C}+\mathcal{B})\otimes I_{n}\big)\zeta $. Given that $ \mathcal{C} $ and $ \mathcal{B} $ are both positive diagonal matrices, there exists a positive vector $ \varsigma = (\mathcal{C}+\mathcal{B})^{-1}\xi \otimes \mathbf{1}_{N} $, such that
\begin{equation*}
	\begin{aligned}
		\varsigma^{\operatorname{T}}\hat{\zeta} = & \big(\xi^{\operatorname{T}}(\mathcal{C}+\mathcal{B})^{-1}\otimes \mathbf{1}_{N}^{\operatorname{T}}\big)\cdot \big((\mathcal{C}+\mathcal{B})\otimes I_{n}\big)\zeta \\[1mm]
		= & \big(\xi^{\operatorname{T}}\mathcal{L}\otimes \mathbf{1}_{N}^{\operatorname{T}}\big)\tilde{x} = 0.
	\end{aligned}
\end{equation*} 
Thus, it follows from \romannumeral2) of Lemma \ref{graph theory lemma} that
\begin{equation*}
	\begin{aligned}
		& -2 \zeta^{\operatorname{T}}\big((\mathcal{C}+\mathcal{B})\mathcal{R}\mathcal{L}(\mathcal{C}+\mathcal{B})\otimes I_{n}\big)\zeta \\[1mm]
		= & - \zeta^{\operatorname{T}}\big((\mathcal{C}+\mathcal{B})(\mathcal{R}\mathcal{L}+\mathcal{L}^{\operatorname{T}}\mathcal{R})(\mathcal{C}+\mathcal{B})\otimes I_{n}\big)\zeta \\[1mm]
		\leq & -\tfrac{\lambda_{2}(\bar{\mathcal{L}})}{N} \zeta^{\operatorname{T}}\big((\mathcal{C}+\mathcal{B})^{2}\otimes I_{n}\big)\zeta.
	\end{aligned}
\end{equation*}
Since it is proved in (\ref{w_i^i})--(\ref{w_i^i_1}) that $ w_i^i(t)>0, ~i=1,2,\ldots,N $ are bounded for all $ t>0 $, $ \check{w}=\min\big\{\inf_{t>0}w_i^i(t), ~i=1,2,\ldots,N \big\}$ is well defined. Then, the following two inequalities are satisfied,
\begin{equation*}
	\left\{\begin{array}{l}
		-2 \zeta^{\operatorname{T}}\big((\mathcal{C}+\mathcal{B})\mathcal{R}\mathcal{L}\mathcal{W}^{-1}\otimes I_{n}\big)h \\[1mm]
		\leq  \tfrac{\lambda_{2}(\bar{\mathcal{L}})}{4N}\zeta^{\operatorname{T}}\big((\mathcal{C}+\mathcal{B})^{2}\otimes I_{n}\big)\zeta + \tfrac{4N\bar{\lambda}_{N}(\mathcal{L}^{\operatorname{T}}\mathcal{L})}{\lambda_{2}(\bar{\mathcal{L}})\check{w}^{2}}\|h\|^{2}, \\[4mm]
		-2 \zeta^{\operatorname{T}}\big((\mathcal{C}+\mathcal{B})\mathcal{R}\mathcal{L}\otimes I_{n}\big)\eta \\[1mm]
		\leq  \tfrac{\lambda_{2}(\bar{\mathcal{L}})}{4N}\zeta^{\operatorname{T}}\big((\mathcal{C}+\mathcal{B})^{2}\otimes I_{n}\big)\zeta + \tfrac{4N\bar{\lambda}_{N}(\mathcal{L}^{\operatorname{T}}\mathcal{L})}{\lambda_{2}(\bar{\mathcal{L}})}\|\eta\|^{2}.
	\end{array}\right.
\end{equation*}
For convenience, we subsequently abbreviate $ \lambda_{2}(\bar{\mathcal{L}}) $ and $ \bar{\lambda}_{N}(\mathcal{L}^{\operatorname{T}}\mathcal{L}) $ as $ \lambda_{2} $ and $ \bar{\lambda}_{N} $, respectively. Then, (\ref{derivative3 of V_1+V_2}) can be rewritten as follows,
\begin{align} \label{derivative4 of V_1+V_2}
	\dot{V}_{1}+\dot{V}_{2} \leq & -\tfrac{\lambda_{2}}{2N} \zeta^{\operatorname{T}}\Big(\big(\mathcal{C}+\mathcal{B}\big)^{2}\otimes I_{n}\Big)\zeta + \tfrac{4N\bar{\lambda}_{N}}{\lambda_{2}\check{w}^{2}}\|h\|^{2}\notag\\
	& + \tfrac{4N\bar{\lambda}_{N}}{\lambda_{2}}\|\eta\|^{2} + \zeta^{\operatorname{T}}\big((\mathcal{C}+\mathcal{R}\mathcal{B}-\sigma_{0}I_{N})\otimes I_{n}\big)\zeta. 
\end{align}

The derivative of $ V_{3} $ along the trajectories of the unperturbed system $ \dot{\chi}= \varphi(\chi) $ is given as follows,
\begin{align} \label{initial_derivative of V_3}
	\dot{V}_{3} = & \big(\zeta+\eta\big)^{\operatorname{T}}\big(\mathcal{R}\otimes I_{n}\big)\big(\dot{\zeta}+\dot{\eta}\big) \notag\\
	= & -\zeta^{\operatorname{T}}\big(\mathcal{R}\mathcal{L}\mathcal{W}^{-1}\otimes I_{n}\big)h - \zeta^{\operatorname{T}}\big(\mathcal{R}\mathcal{L}\otimes I_{n}\big)\eta \notag\\
	& -\eta^{\operatorname{T}}\big(\mathcal{R}\mathcal{L}\mathcal{W}^{-1}\otimes I_{n}\big)h - \eta^{\operatorname{T}}\big(\mathcal{R}\mathcal{L}\otimes I_{n}\big)\eta. 
\end{align}
Applying \romannumeral2) of Lemma \ref{graph theory lemma} leads to
\begin{equation} \label{inequality1 of initial_derivative of V_3}
	- \eta^{\operatorname{T}}\big(\mathcal{R}\mathcal{L}\otimes I_{n}\big)\eta \leq -\tfrac{\lambda_{2}}{2}\|\eta\|^{2}.
\end{equation}
Moreover, it can be verified that the following inequalities hold,
\begin{equation}\label{inequality2 of initial_derivative of V_3}
	\left\{\begin{array}{l}
		\!\!- \zeta^{\operatorname{T}}\big(\mathcal{R}\mathcal{L}\otimes I_{n}\big)\eta \leq  \tfrac{\lambda_{2}}{8}\|\eta\|^{2} +
		\tfrac{2\bar{\lambda}_{N}}{\lambda_{2}}\|\zeta\|^{2},  \\[2mm]
		\!\!-\eta^{\operatorname{T}}\big(\mathcal{R}\mathcal{L}\mathcal{W}^{-1}\otimes I_{n}\big)h \leq  \tfrac{\lambda_{2}}{8}\|\eta\|^{2} + \tfrac{2\bar{\lambda}_{N}}{\lambda_{2}\check{w}^{2}}\|h\|^{2}, \\[2mm]
		\!\!-\zeta^{\operatorname{T}}\big(\mathcal{R}\mathcal{L}\mathcal{W}^{-1}\otimes I_{n}\big)h \leq  \|\zeta\|^{2} + \tfrac{\bar{\lambda}_{N}}{4\check{w}^{2}}\|h\|^{2}. 
	\end{array}\right.
\end{equation}
Then, substituting inequalities (\ref{inequality1 of initial_derivative of V_3})--(\ref{inequality2 of initial_derivative of V_3})  into (\ref{initial_derivative of V_3}) yields 
\begin{align} \label{final_derivative of V_3}
	\dot{V}_{3} \leq  -\tfrac{\lambda_{2}}{4}\|\eta\|^{2} + \tfrac{\lambda_{2}+2\bar{\lambda}_{N}}{\lambda_{2}}\|\zeta\|^{2} + \tfrac{(8+\lambda_{2})\bar{\lambda}_{N}}{4\lambda_{2}\check{w}^{2}}\|h\|^{2}.
\end{align}
By combining (\ref{derivative4 of V_1+V_2}) and (\ref{final_derivative of V_3}), the derivative of $ V $ in (\ref{Lyapunov function candidate}) along the trajectories of the unperturbed system $ \dot{\chi}= \varphi(\chi) $ satisfies the following inequality,
\begin{align} \label{derivative1 of V}
	\dot{V} \leq & \zeta^{\operatorname{T}}\Big(\big(-\tfrac{\lambda_{2}}{2N}(\mathcal{C}+\mathcal{B})^{2}+(\mathcal{C}+\mathcal{R}\mathcal{B})-\sigma_{0}I_{N} \notag\\
	& +\tfrac{33N\bar{\lambda}_{N}(\lambda_{2} +2\bar{\lambda}_{N})}{\lambda_{2}^{3}}I_{N} \big)\otimes I_{n}\Big)\zeta - \tfrac{17N\bar{\lambda}_{N}}{4\lambda_{2}}\|\eta\|^{2} \notag\\
	& + \Big(\tfrac{4N\bar{\lambda}_{N}}{\lambda_{2}\check{w}^{2}} + \tfrac{33N\bar{\lambda}_{N}^{2}(8+\lambda_{2})}{4\lambda_{2}^{3}\check{w}^{2}}\Big) \|h\|^{2}. 
\end{align}
By the Lipschitz condition of the gradients in Assumption \ref{assumption_cost functions}, we have $ h= \nabla \tilde{f}(\bar{x}+\tilde{x}) - \nabla \tilde{f}(\bar{x}) \leq \hat{l}\|\tilde{x}\| $, where $ \hat{l} = \max\{l_{1},l_{2},\ldots,l_{N}\} $. 

Denote the second smallest eigenvalue of matrix $ \mathcal{L}^{\operatorname{T}}\mathcal{L} $ by $ \bar{\lambda}_{2}(\mathcal{L}^{\operatorname{T}}\mathcal{L}) $, or simply $ \bar{\lambda}_{2} $. Recall that $ \zeta = (\mathcal{L}\otimes I_{n})\tilde{x}$ and $ \mathcal{L}\mathbf{1}_{N}=\mathbf{0}_{N}$. It can be obtained from the symmetry of $ \mathcal{L}^{\operatorname{T}}\mathcal{L} $ and $ \bar{\lambda}_{2}(\mathcal{L}^{\operatorname{T}}\mathcal{L})>0 $ that $ \|\zeta\|^{2}= \tilde{x}^{\operatorname{T}}\big(\mathcal{L}^{\operatorname{T}}\mathcal{L}\otimes I_{n}\big)\tilde{x}\geq \bar{\lambda}_{2}\|\tilde{x}\|^{2} $. Thus, one has $ \|h\|^{2}\leq \hat{l}^{2}/\bar{\lambda}_{2}\|\zeta\|^{2} $.
Define $ \omega_{1} = \tfrac{33N\bar{\lambda}_{N}(\lambda_{2} +2\bar{\lambda}_{N})}{\lambda_{2}^{3}} $ and $\omega_{2} = \frac{\hat{l}^{2}}{\bar{\lambda}_{2}}\Big( \tfrac{4N\bar{\lambda}_{N}}{\lambda_{2}\check{w}^{2}} + \tfrac{17N\bar{\lambda}_{N}^{2}(8+\lambda_{2})}{4\lambda_{2}^{3}\check{w}^{2}}\Big) $. Then, (\ref{derivative1 of V}) can be rewritten as follows,
\begin{align} \label{derivative of V}
	\dot{V} \leq & - \zeta^{\operatorname{T}}\Big(\tfrac{\lambda_{2}}{2N}\big((\mathcal{C}+\mathcal{B})-\tfrac{N}{\lambda_{2}}I_{N}\big)^{2}\otimes I_{n}\Big)\zeta  \notag\\
	& -\Big(\sigma_{0} -\omega_{1}-\omega_{2} -\tfrac{N}{2\lambda_{2}} \Big) \|\zeta\|^{2} - \tfrac{17N\bar{\lambda}_{N}}{4\lambda_{2}}\|\eta\|^{2}. 
\end{align}
Choose $ \sigma_{0} = 1+ \omega_{1} + \omega_{2} + \tfrac{N}{2\lambda_{2}} $. One thus has
\begin{align}
	\dot{V} \leq & -\|\zeta\|^{2} - \tfrac{17N\bar{\lambda}_{N}}{4\lambda_{2}}\|\eta\|^{2}. 
\end{align}
Therefore, it is proved that the equilibrium point $\chi=\mathbf{0}$ of the unperturbed system $ \dot{\chi}= \varphi(\chi) $ is globally uniformly asymptotically stable, and the adaptive control gains $ \sigma_{i},~i=1,2,\ldots,N $ converge to some finite positive constants.

\textit{Step 2.} Show that the perturbed system (\ref{tilde_x_new dynamics}) is ISS, and the state $ \chi(t) $ converges to zero as $ t\to\infty $.
Reconsider the Lyapunov function candidate $ V $ in (\ref{Lyapunov function candidate}), but with $ \sigma_{0} $ being another positive constant to be specified. Similarly, by referring to (\ref{derivative of V}), the derivative of $ V $ along the trajectories of (\ref{tilde_x_new dynamics}) can be given as follows,
\begin{equation} \label{dot_V_per1}
	\small
	\begin{aligned} 
		\dot{V} \leq &  - \zeta^{\operatorname{T}}\Big(\tfrac{\lambda_{2}}{2N}\big(\mathcal{C}+\mathcal{B}-\tfrac{N}{\lambda_{2}}I_{N}\big)^{2}\otimes I_{n}\Big)\zeta \\
		& -\Big(\sigma_{0} -\omega_{1}-\omega_{2} -\tfrac{N}{2\lambda_{2}} \Big) \|\zeta\|^{2} - \tfrac{17N\bar{\lambda}_{N}}{4\lambda_{2}}\|\eta\|^{2} \\
		& + 2 \zeta^{\operatorname{T}}\Big(\big(\mathcal{C}+\mathcal{B}\big)\mathcal{R}\mathcal{L}(\mathcal{R}^{-1}-\mathcal{W}^{-1}(t))\otimes I_{n}\Big)\nabla \tilde{f}(\bar{x}) \\
		& + \epsilon \big(\zeta+\eta\big)^{\operatorname{T}}\Big(\mathcal{R}\mathcal{L}\big(\mathcal{R}^{-1}-\mathcal{W}^{-1}(t)\big)\otimes I_{n}\Big)\nabla \tilde{f}(\bar{x}),
	\end{aligned}
\end{equation}
where $ \epsilon = \tfrac{33N\bar{\lambda}_{N}(\mathcal{L}^{\operatorname{T}}\mathcal{L})}{\lambda_{2}(\bar{\mathcal{L}})^{2}} $.
Define $ u(t)= \big(\mathcal{R}\mathcal{L}(\mathcal{R}^{-1}-\mathcal{W}^{-1}(t))\otimes I_{n}\big)\nabla \tilde{f}(\bar{x}) $. Then the following two inequalities are satisfied,
\begin{equation} \label{inequalities_dot_V_per1}
	\left\{
	\begin{aligned}
		2 \zeta^{\operatorname{T}}\big((\mathcal{C}&+\mathcal{B})\mathcal{R}\mathcal{L}\big(\mathcal{R}^{-1}-\mathcal{W}^{-1}(t)\big)\otimes I_{n}\big)\nabla \tilde{f}(\bar{x}) \\
		\leq & \tfrac{\lambda_{2}}{4N}\zeta^{\operatorname{T}}\big((\mathcal{C}+\mathcal{B})^{2}\otimes I_{n}\big)\zeta + \tfrac{4N}{\lambda_{2}}\|u(t)\|^{2},  \\[1.5mm]
		\big(\zeta+\eta\big)^{\operatorname{T}}&\big(\mathcal{R}\mathcal{L}(\mathcal{R}^{-1}-\mathcal{W}^{-1}(t))\otimes I_{n}\big)\nabla \tilde{f}(\bar{x}) \\
		\leq & \tfrac{\lambda_{2}}{8}\|\eta\|^{2} + \|\zeta\|^{2} + \tfrac{8+\lambda_{2}}{4\lambda_{2}}\|u(t)\|^{2}.
	\end{aligned}
	\right.
\end{equation}
By substituting (\ref{inequalities_dot_V_per1}) into (\ref{dot_V_per1}), one has
\begin{equation*}
	\begin{aligned} 
		\dot{V} \leq &  - \zeta^{\operatorname{T}}\Big(\tfrac{\lambda_{2}}{4N}\big(\mathcal{C}+\mathcal{B}-\tfrac{2N}{\lambda_{2}}I_{N}\big)^{2}\otimes I_{n}\Big)\zeta  -\Big(\sigma_{0} -\omega_{1} -\omega_{2} \\
		& - \epsilon -\tfrac{N}{\lambda_{2}} \Big) \|\zeta\|^{2} - \tfrac{N\bar{\lambda}_{N}}{8\lambda_{2}}\|\eta\|^{2} + \Big( \tfrac{4N}{\lambda_{2}} + \tfrac{\epsilon(8+\lambda_{2})}{4\lambda_{2}} \Big)\|u(t)\|^{2}.
	\end{aligned}
\end{equation*}

Define $ \varrho= \tfrac{4N}{\lambda_{2}} + \tfrac{\epsilon(8+\lambda_{2})}{4\lambda_{2}} $, and $ \kappa=\min\big\{\tfrac{N\bar{\lambda}_{N}}{8\lambda_{2}}, 1\big\} $.
By choosing $ \sigma_{0} = 1+ \omega_{1} + \omega_{2} + \epsilon + \frac{N}{\lambda_{2}} $, one has
\begin{align*}
	\dot{V} \leq & -\|\zeta\|^{2} - \tfrac{N\bar{\lambda}_{N}}{8\lambda_{2}}\|\eta\|^{2} +  + \Big( \tfrac{4N}{\lambda_{2}} + \tfrac{\epsilon(8+\lambda_{2})}{4\lambda_{2}} \Big)\|u(t)\|^{2} \\
	\leq & \kappa\|\chi(t)\|^{2} +  \varrho\|u(t)\|^{2}.
\end{align*}
Let $ 0<\theta<1 $, it can be obtained that
\begin{align*}
	\dot{V} \leq & -(1-\theta)\kappa\|\chi(t)\|^{2}, \quad \forall~ \|\chi(t)\| \geq \sqrt{\frac{\varrho}{\theta\kappa}}\|u(t)\|.
\end{align*}
It then follows from Theorem 4.19 in \cite{khalil2002nonlinear} that the system (\ref{tilde_x_new dynamics}) is ISS.

Note that it is proved in (\ref{w_i^i_1}) that $ \lim_{t\to\infty}w(t)=\mathbf{1}_{N}\otimes \xi  $. One then has $ \lim_{t\to\infty} w_{i}^{i}(t) = \xi_{i} $ and $ \lim_{t\to\infty} \mathcal{W}^{-1}(t) = \mathcal{R}^{-1} $. Thus, $ u(t)= \big(\mathcal{R}\mathcal{L}(\mathcal{R}^{-1}-\mathcal{W}^{-1}(t))\otimes I_{n}\big)\nabla \tilde{f}(\bar{x}) $ converges to zero as $ t\to\infty $. By the definition of ISS, it can be proved (see Exercise 4.58 in \cite{khalil2002nonlinear}) that the state $ \chi(t)=\operatorname{col}(\zeta(t), \eta(t)) $ converges to zero as time goes to infinity. 
\vskip 0.15cm
\noindent\textbf{Part 2.} Show that $ \lim_{t\to\infty}\tilde{x}(t)= \mathbf{0}$ and $ \lim_{t\to\infty}\tilde{v}(t)= \mathbf{1}_{N}\otimes \tau_{2}$, for a constant vector $ \tau_{2}\in\mathbb{R}^{n} $. 

Note that $ \zeta = (\mathcal{L}\otimes I_{n})\tilde{x} $ and $ \eta = (\mathcal{L}\otimes I_{n})\tilde{v} $. By the obtained facts that $ \lim_{t\to\infty} \zeta(t)=\mathbf{0} $ and $ \lim_{t\to\infty} \eta(t)=\mathbf{0} $, one has $ \lim_{t\to\infty}\tilde{x}(t)= \mathbf{1}_{N}\otimes \tau_{1}$ and $ \lim_{t\to\infty}\tilde{v}(t)= \mathbf{1}_{N}\otimes \tau_{2}$, for two constant vectors $ \tau_{1},\tau_{2}\in\mathbb{R}^{n} $. Next, we will show that $ \tau_{1}=\mathbf{0}_{n} $ and $ \tau_{2}<\infty $ by seeking a contradiction.

To this end, by taking limits on both sides of the equation in (\ref{compact_form_2_a}), one has 
\begin{equation} \label{equilibrium_point}
	\small
	\begin{aligned}
		\mathbf{0} = & \lim_{t\to\infty} \Big(-(\mathcal{W}^{-1}(t)\otimes I_{n})\nabla \tilde{f}(\bar{x}+\tilde{x}(t))  + \big(\mathcal{R}^{-1} \otimes I_{n}\big)\nabla \tilde{f}(\bar{x})\Big) \\
		= & \big(\mathcal{R}^{-1} \otimes I_{n}\big) \big(\nabla \tilde{f}(\bar{x}) - \nabla \tilde{f}(\bar{x} + \mathbf{1}_{N}\otimes \tau_{1}) \big).
	\end{aligned}
\end{equation} 
By pre-multiplying both sides of equation (\ref{equilibrium_point}) with $ \xi^{\mathrm{T}} \otimes I_{n} $, one then has $ \sum_{i=1}^{N}\nabla f_{i}(\bar{x}_{i}) = \sum_{i=1}^{N}\nabla f_{i}(\bar{x}_{i} + \tau_{1}) $. Under Assumption \ref{assumption_cost functions}, the optimal solution to the DOP in (\ref{problem}) is unique, which implies that $ \tau_{1}=\mathbf{0}_{n} $. Meanwhile, it follows from (\ref{compact_form_2_b}) that $ \lim_{t\to\infty} \dot{\tilde{v}}(t) = \mathbf{0} $. One thus has $ \tau_{2}<\infty $. To sum up, the trajectories of (\ref{compact_form_2}) are bounded for all $ t>0 $, and $ \tilde{x}(t) $ tends to zero as time goes to infinity. 
Therefore, the convergence of the states $ x_{i}, ~i=1,2,\ldots,N $ of (\ref{algorithm}) to the optimal solution  $ s^{\star} $ of the problem (\ref{problem}) is presented, and the proof is thus completed.

\end{proof}

\begin{remark}
	The adaptive gain $ \sigma_{i} $ in (\ref{algorithm}) is updated based on relative state errors so that it will always increase as long as the consensus is not achieved, eventually rendering the state consensus of the agents. By adopting the adaptive control approach, the requirement on the smallest strong convexity constant of local cost functions is no longer needed to generate a negative term related to agent states. Thus, the proposed adaptive algorithm is able to sovle the DOP when local cost functions are nonconvex.
\end{remark}

\begin{remark}
	The DOP over unbalanced digraphs is investigated in \cite{Zhu2018continuous}. The advantages of the distributed adaptive algorithm (\ref{algorithm}) in this work over the algorithm in \cite{Zhu2018continuous} can be stated in the following two aspects. First, the control gains in \cite{Zhu2018continuous} rely on some prior global information concerning network connectivity and cost functions, such as the second smallest eigenvalue of the Laplacian matrix, the smallest strong convexity constant of local cost functions as well as the largest Lipschitz constant of their gradients. On the contrary, the algorithm (\ref{algorithm}) developed in this work does not require those information and is thus fully distributed. Second, to guarantee the convergence of their algorithm, sufficient large control gains are needed in \cite{Zhu2018continuous}. A potential problem with high-gain feedback is that it may result in some undesirable issues, such as increased sensitivity to unmodeled dynamics and noise, oscillations, and even instability. The algorithm developed in this work can avoid this problem by adopting an adaptive control mechanism.
\end{remark}

\section{Illustrative Examples} \label{section simulation results}

In this section, the effectiveness of the developed fully distributed algorithm (\ref{algorithm}) over unbalanced directed networks is illustrated by two examples. 

\subsection{Example 1}
Consider five networked agents with their communication network topology being described by the unbalanced digraph $ \mathcal{G} $ in Fig. \ref{Fig_topology1}. It can be verified that this digraph is strongly connected, and Assumption \ref{graph assumption} is thus satisfied.
Suppose that agent $ i,~i=1,\ldots,5 $, are endowed with the following local cost functions respectively:
\begin{align*}
	f_{1} & =5\sin\big(\|s+[4, 5]^{\mathrm{T}}\|\big), ~ f_{2}=10\cos\big(\ln(\|s+[8, 10]^{\mathrm{T}}\|)\big), \\
	f_{3} & =4\times\|s+[2, 3]^{\mathrm{T}}\|^{\frac{4}{3}}, \quad f_{4}=2\times\|s-[3, 5]^{\mathrm{T}}\|^{2},\\
	f_{5} & =\|s+[1, 2]^{\mathrm{T}}\|^{2}/\sqrt{\|s+[1, 2]^{\mathrm{T}}\|^{2}+2 },
\end{align*}
where $ s\in\mathbb{R}^{2} $.
Note that the local cost functions $ f_{1}(\cdot) $ and $ f_{2}(\cdot) $ are nonconvex. However, it can be verified that the global cost function $ f(s)=\sum_{i=1}^{5}f_{i}(s) $ is strictly convex, which implies that the global minimizer $ s^{\star} $ is unique. Moreover, it can be verified that the gradients of $ f_{i}(\cdot), ~i=1,\ldots,5 $ are globally Lipschitz on $ \mathbb{R} $, and Assumption \ref{assumption_cost functions} is thus satisfied.

\begin{figure}[!t] 
	\centering 
	\begin{tikzpicture}[> = stealth, 
	shorten > = 1pt, 
	auto,
	node distance = 3cm, 
	semithick 
	,scale=0.6,auto=left,every node/.style={circle,fill=gray!30,draw=black!80,text centered}]
	\centering
	\node (n1) at (0,0)		{1};
	\node (n2) at (2,0)  	{2};
	\node (n3) at (4,0) 	{3};
	\node (n4) at (6,0) 	{4};
	\node (n5) at (8,0) 	{5};
	
	\draw[->,black!80] (n3) to [out=-135,in=-45] (n1);
	\draw[->,black!80] (n1)-- (n2);
	\draw[->,black!80] (n2)-- (n3);
	\draw[->,black!80] (n3)-- (n4);
	\draw[->,black!80] (n4)-- (n5);
	\draw[->,black!80] (n2) to [out=-35,in=-145] (n5);
	\draw[->,black!80] (n5) to [out=150,in=30] (n1);
	
	\end{tikzpicture} 
	\caption{An unbalanced directed network.} 
	\label{Fig_topology1}
\end{figure}
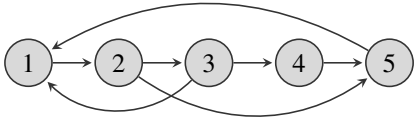

\begin{figure}[t] 
	\centering 
	\includegraphics[height=6.2cm, width=8.4cm]{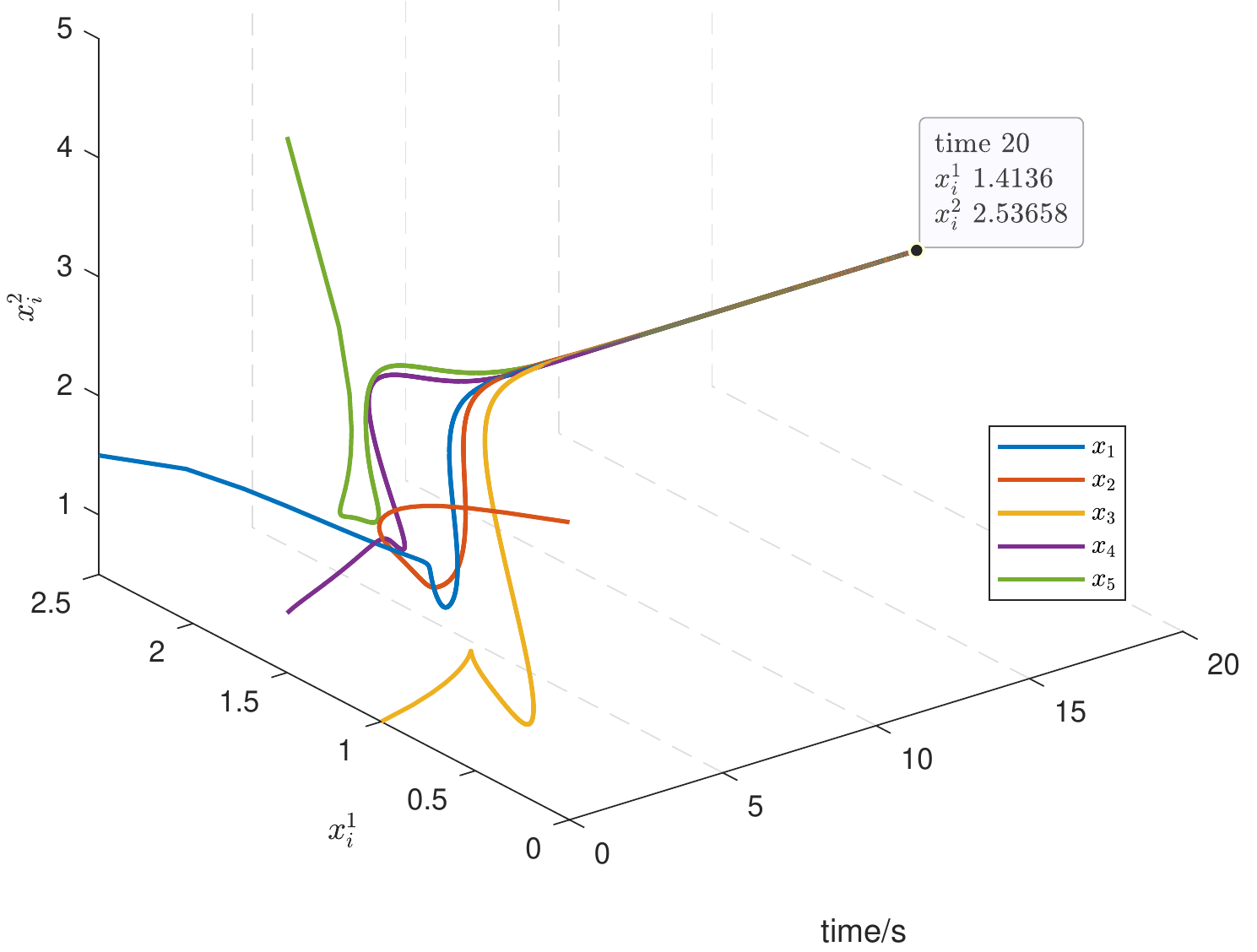} 
	\caption{Convergence performance of the fully distributed algorithm (\ref{algorithm}) over the unbalanced directed network in Fig. \ref{Fig_topology1} under the relaxed condition that the local cost functions are nonconvex.} 
	\label{Fig_compare_on_k_ci}
	\vspace{-0.2cm}
\end{figure}

We now present the convergence performance of the fully distributed algorithm (\ref{algorithm}) under the relaxed condition that the global cost function is strictly convex while local cost functions may be nonconvex. For $ i=1,\ldots,5 $, let the initial values $ x_{i}(0) $ and $ v_{i}(0) $ be arbitrarily chosen, and the initial values of the adaptive gains $ \sigma_{i} $'s be chosen as $ \sigma_{i}(0)=1 $. The simulation results are shown in Fig. \ref{Fig_compare_on_k_ci}. 
It can be observed that the trajectories of agent states $ x_{i}=[x_{i1}, x_{i2}]^{\mathrm{T}}, ~i=1,\ldots,5 $ converge to the optimal solution $ s^{\star}=[1.4136,~ 2.53658]^{\mathrm{T}} $, which minimizes the global cost function $ f(s)=\sum_{i=1}^{5}f_{i}(s) $. Therefore, without either strong convexity of local cost functions or prior global information of network connectivity, it has been shown that the fully distributed algorithm (\ref{algorithm}) solves the DOP in (\ref{problem}).

\subsection{Example 2}
\begin{figure}[t] 
	\centering  
	\includegraphics[height=7.6cm, width=8.8cm]{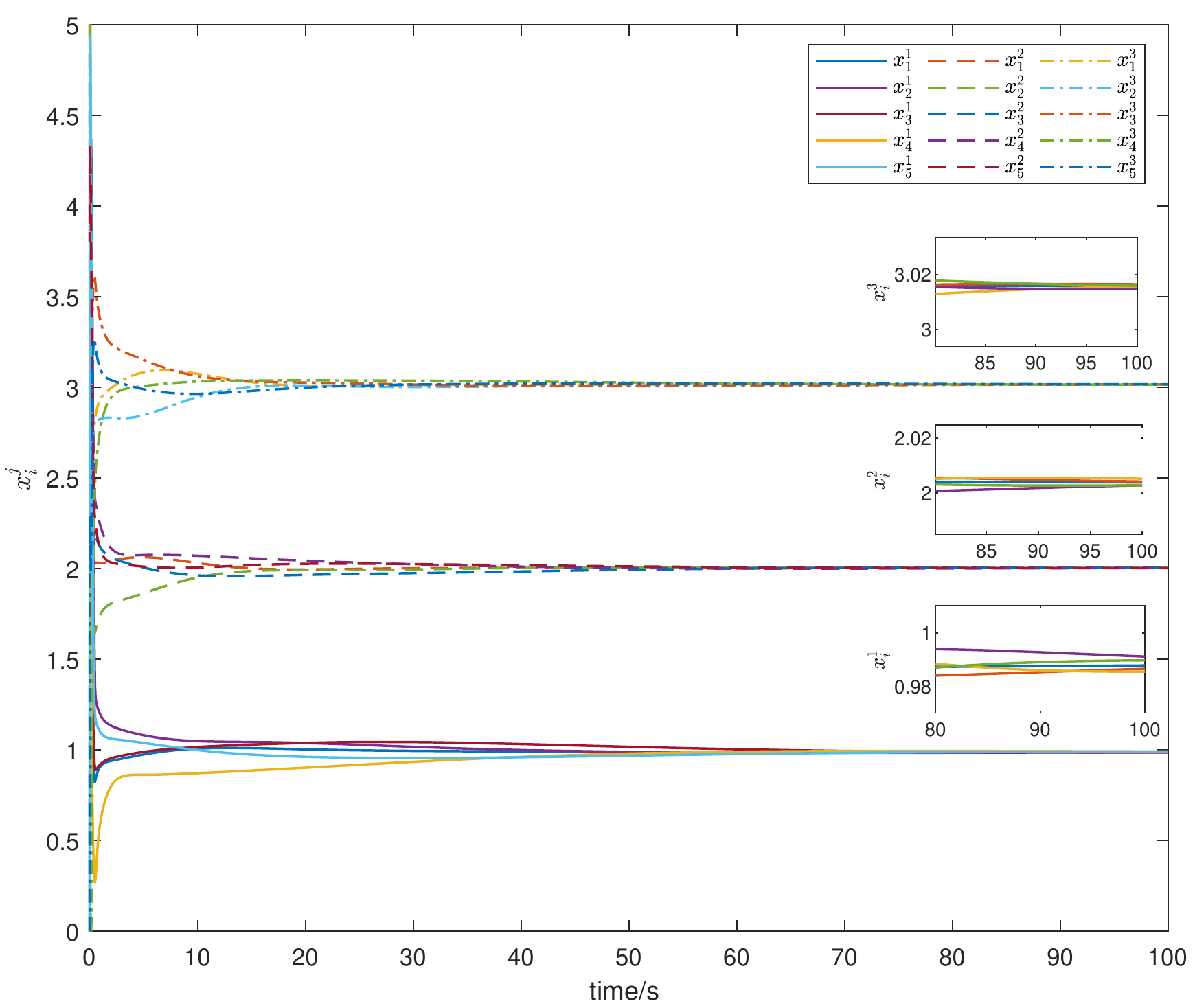} 
	\caption{Convergence performance of algorithm (\ref{algorithm}) for solving the distributed parameters estimation problem. Each measurement satisfies the independent identically Gaussian distribution $ \mathcal{N}(\mu, \sigma_{i}) $, with mean vector $ \mu=[1,2,3]^{\operatorname{T}} $ and covariance matrix $ \sigma_{i} = 0.05iI_{3}, i=1,2,\ldots,5 $. Huber loss parameter $ \varsigma=0.5 $.}
	\label{Fig_compare_existing} 
\end{figure}

In this example, we apply the fully distributed algorithm (\ref{algorithm}) to solve the distributed parameter estimation problem, which is one of the most important research topics in wireless sensor networks. The objective is to estimate a parameter or function based on large amounts of data collected by sensors from the environment. The distributed parameter estimation problem can be reformulated as a DOP. Moreover, the estimator is typically the optimal solution of the corresponding optimization problem. Please refer to \cite{Rabbat2005quantized} and the references therein for more details.

More specifically, consider a group of $ N $ sensors over general directed networks that cooperatively estimate a parameter $ s\in \mathbb{R}^{n} $ by collecting private measurements. Suppose that each sensor $ i $ collects $ n_{i} $ measurements $ Q_{ij}\in \mathbb{R}^{n},j=1,\ldots,n_{i} $. Let $ Q_{i}= \{Q_{ij},j=1,\ldots,n_{i}\} $ denote the private data set of sensor $ i $. Then the local cost function of sensor $ i $ can be defined as follows,
\begin{equation*}
	f_{i}(s)= \big\|\sum_{Q_{ij}\in Q_{i}} H(Q_{ij}, s)\big\|_{1},
\end{equation*}
where $ H(\cdot,\cdot) $ represents the Huber loss function. It is noted that $ f_{i}(s) $ is nondifferentiable. In particular, the Huber loss function for one-dimensional variable is defined as follows,
\begin{equation*}
	H(Q_{ij}, s)=\left\{\begin{array}{l}
		\tfrac{\left(Q_{i j}-s\right)^{2}}{2}, \text { for }\left|Q_{i j}-s\right| \leq \varsigma, \\
		\varsigma\left|Q_{i j}-s\right|-\tfrac{\varsigma^{2}}{2}, \text { for }\left|Q_{i j}-s\right|>\varsigma,
	\end{array}\right.
\end{equation*}
where $ \varsigma $ is a positive constant. Compared with the typical least squares loss function, the Huber loss function takes a smaller value when the difference between the measurements and the parameter to be estimated exceeds the tolerance $ \varsigma $. Therefore, it is less sensitive to outliers and improves the robustness of the estimators. The corresponding DOP can be formulated as follows,
\begin{equation} \label{optimization_estimation}
	\min_{s\in \mathbb{R}^{n}} f(s) = \sum_{i=1}^{N} \big\|\sum_{Q_{ij}\in Q_{i}} H(Q_{ij}, s)\big\|_{1}.
\end{equation}

To proceed, consider a sensor network with communication topology depicted by Fig. \ref{Fig_topology1}. Let the parameter to be estimated be taken as $ s= [1,2,3]^{\operatorname{T}} $. Assume that the data set of each sensor $ i $ contains $ 500 $ measurements. Besides, the obtained measurements satisfy the independent identically Gaussian distribution $ \mathcal{N}(\mu, \sigma_{i}) $ with mean vector $ \mu=[1,2,3]^{\operatorname{T}} $ and covariance matrix $ \sigma_{i} = 0.01iI_{3}, i=1,\ldots,5 $. Let the Huber loss parameter $ \varsigma =0.5 $. We apply the proposed fully distributed algorithm (\ref{algorithm}) to solve the distributed parameter estimation problem in (\ref{optimization_estimation}). Let the initial values of adaptive gains $ \sigma_{i} $ be chosen as $ \sigma_{i}(0)=1$, for $ i=1,\ldots,5 $. The initial values of agent states $ x_{i} $ and auxiliary variables $ v_{i} $ are randomly chosen. 

The simulation result is shown in Fig. \ref{Fig_compare_existing}. It can be observed that $ |x_{i}^{1} - 1| $, $ |x_{i}^{2} - 2| $ and $ |x_{i}^{3} - 3|, i=1,2,\ldots,5 $ are all upper bounded by 0.02 before 100s. In other words, the trajectories of states $ x_{i}, i=1,\ldots,5 $ converge to a small neighborhood of $ s= [1,2,3]^{\operatorname{T}} $, which is the parameter to be estimated. Therefore, it is shown that the distributed parameter estimation problem over unbalanced directed networks can be solved by the fully distributed algorithm  (\ref{algorithm}).


\section{Conclusion} \label{section conclusion}
In this paper, a novel fully distributed continuous-time algorithm has been developed to solve the distributed nonconvex optimization problem over unbalanced directed networks. The proposed adaptive algorithm design does not need any prior global information concerning network connectivity or convexity of local cost functions. Under mild assumptions, the proposed algorithm has guaranteed that the agent states converge to the optimal solution that minimizes the sum of the local cost functions. The effectiveness of the developed fully distributed algorithm has been illustrated by two simulation examples. Future work will focus on solving the fully distributed optimization problem for systems with more general agent dynamics.

%

\bibliographystyle{IEEEtran}  
\bibliography{IEEEabrv,mylib}

\end{document}